\documentclass[12pt,letterpaper]{article}

\usepackage{amssymb,amsfonts,amscd,amsthm}
\usepackage[all,arc]{xy}
\usepackage{enumerate, bbm}
\usepackage{mathrsfs}
\usepackage{tikz}
\usepackage[left=0.9in,top=0.9in,right=0.9in,bottom=0.9in]{geometry}
\usepackage{mathtools}
\usepackage{hyperref}
\usepackage{color}
\usepackage{graphicx}
\usepackage{enumitem} 
\graphicspath{ {TexImages/} }

\newtheorem{thm}{Theorem}[section]
\newtheorem{cor}[thm]{Corollary}
\newtheorem{prop}[thm]{Proposition}
\newtheorem{lem}[thm]{Lemma}

\theoremstyle{definition}

\newtheorem{quest}[thm]{Question}

\newtheorem{conj}[thm]{Conjecture}

\hypersetup{
	colorlinks,
	citecolor=blue,
	filecolor=blue,
	linkcolor=blue,
	urlcolor=blue,
	linktocpage
}

\setenumerate[1]{label=\thesection.\arabic*.} 
\setenumerate[2]{label*=\arabic*.} 
\setlist[enumerate]{itemsep=2ex, topsep=2ex} 
\setlist[itemize]{itemsep=2ex, topsep=2ex}

\newcommand{\R}{\mathbb{R}}
\newcommand{\E}{\mathbb{E}}

\newcommand{\ep}{\epsilon}
\newcommand{\lam}{\lambda}
\renewcommand{\l}{\left}
\renewcommand{\r}{\right}
\newcommand{\Del}{\Delta}
\newcommand{\half}{\frac{1}{2}}

\newcommand{\sm}{\setminus}
\newcommand{\sub}{\subseteq}

\renewcommand{\c}[1]{\mathcal{#1}}
\newcommand{\ol}[1]{\overline{#1}}
\newcommand{\rec}[1]{\frac{1}{#1}}
\newcommand{\f}[2]{\frac{#1}{#2}}
\newcommand{\norm}[1]{\lVert#1\rVert}
\newcommand{\mr}[1]{\mathrm{#1}}

\newcommand{\Av}[3]{\mathbf{\mr{R}}_{#3}^{#2}(#1)}
\newcommand{\AvN}[3]{\mathbf{\tilde{\mr{R}}}_{#3}^{#2}(#1)}
\newcommand{\e}{e}

\title{An Averaging Processes on Hypergraphs}
\author{Sam Spiro\footnote{Dept.\ of Mathematics, UCSD {\tt sspiro@ucsd.edu}.}}
\date{\today}
\begin{document}
\maketitle

\begin{abstract}
	Consider the following iterated process on a hypergraph $H$.  Each vertex $v$ starts with some initial weight $x_v$.  At each step, uniformly at random select an edge $e$ in $H$, and for each vertex $v$ in $e$ replace the weight of $v$ by the average value  of  the vertex weights over  all vertices in $e$. This is a generalization of an interactive process on graphs which was first introduced by  Aldous and Lanoue.   In this paper  we use the eigenvalues of a Laplacian for hypergraphs to bound the rate of convergence for this iterated averaging process.
\end{abstract}
\section{Introduction}
The following iterated process on a graph $G$ was introduced by Aldous and Lanoue~\cite{AL}. 

{\it Initially, assign $n$ real numbers $x_1,\ldots,x_n$ to the vertices of $G$.  At each step uniformly at random select $\{i,j\}\in E(G)$ and replace both $x_i$ and $x_j$ with their average value $(x_i+x_j)/2$.  }

As noted in \cite{AL},  this process was motivated by the study of social dynamics and  interactive particle systems.
Recently Chaterjee, Diaconis, Sly, and Zhang~\cite{CD} further investigated this process  in response to a question of Bourgain and a problem arising in quantum computing. In particular they obtained sharper estimates for the rate of convergence for the case of $G$ being a complete graph.

There are various procedures  similar to the above process, such as the gossip algorithms studied by Shah~\cite{S},  the distributed consensus algorithms studied by Olshevsky and Tsitsiklis~\cite{OT}, and various restricted averaging processes 
 \cite{AC, Ben} as well as numerous `smoothing' or `renewal'  models in statistics \cite{Cha, Fe}. In addition, there are numerous random processes sharing similar flavors and methods, such as exchanging processes on permutations and card shuffling \cite{Dia}.

In this paper, we consider the following averaging process on a hypergraph $H$, and we emphasize that we place no restriction on the multiplicity or size of any edge of $H$.

{\it Initially, assign $n$ real numbers $x_1,\ldots,x_n$ to the vertices of $H$.  At each step uniformly at random select an edge $\e\in E(H)$, and for each $i\in \e$ replace $x_i$ with the average value $|\e|^{-1}\sum_{j\in \e} x_j$.} 

Let us formalize this process somewhat. Recall that a \textit{hypergraph} $H$ is a set of vertices $V(H)$ together with a multiset $E(H)$ of subsets of $V(H)$ which are called edges.   Given a hypergraph $H$, let $x$ be a real-valued vector indexed by $V(H)$, which we call a \textit{weight vector} of $H$.  Define the (random) vector $\Av{x}{}{H}$ by choosing an edge $\e$ uniformly at random from $E(H)$, and then setting $\Av{x}{}{H}_u=x_u$ if $u\notin \e$ and $\Av{x}{}{H}_u=|\e|^{-1}\sum_{v\in \e} x_v$ otherwise.   Recursively define $\Av{x}{t}{H}=\Av{\Av{x}{t-1}{H}}{}{H}$.  Equivalently, $\Av{x}{t}{H}$ is the random vector obtained by uniformly generating a sequence of $t$ edges and then performing the averaging process for each edge sequentially.  When $H$ is understood we simply write $\Av{x}{t}{}$.

Given a weight vector $x$ of $H$ with $|V(H)|=n$, define the vector $\ol{x}=(\rec{n}\sum x_u,\ldots,\rec{n}\sum x_u)$.   We wish to determine how quickly $\Av{x}{t}{}$ converges to $\ol{x}$ in various norms. In the graph setting, Aldous and Lanoue~\cite{AL} bounded this rate of convergence in terms of the second smallest eigenvalue of the (combinatorial) Laplacian.  There are many ways to generalize the Laplacian for hypergraphs~\cite{H1,H2,H3}, and in this paper we use a generalization which was introduced by Rodr\'{i}guez \cite{R}.

Given a hypergraph $H$, the \textit{codegree} $d(u,v)$ of two vertices $u\ne v$ is defined to be the number of edges $\e$ containing both $u$ and $v$ in $H$.  If $H$ is an $n$-vertex hypergraph, the \textit{codegree Laplacian} $L(H)$ is the $n\times n$ matrix with $L(H)_{u,v}=-d(u,v)$ if $u\ne v$ and $L(H)_{u,u}=\sum_{v\ne u} d(u,v)$. For example, if $n=4$ and $E(H')=\{\{1,2\},\{1,2\},\{2,3,4\}\}$, then \[L(H')=\begin{bmatrix}
2 & -2 & 0 & 0\\ 
-2 & 4 & -1 & -1\\ 
0 & -1 & 2 & -1\\ 
0 & -1 & -1 & 2
\end{bmatrix}.\]
Note that when $H$ is a graph this reduces to the Laplacian matrix of $H$.  In fact, $L(H)$ can be defined to be the Laplacian for the multi-graph $G_H$ obtained by placing a clique on all of the vertices of each $\e\in E(H)$.  For example, with $H'$ as above, $G_{H'}$ is the multi-graph displayed below.

\begin{center}
\begin{tikzpicture}[scale=0.4]

\node at (0,0) {$\bullet$};
\node at (0,1) {1};
\node at (4,0) {$\bullet$};
\node at (4,1) {2};
\node at (6,2) {$\bullet$};
\node at (7,2) {3};
\node at (6,-2) {$\bullet$};
\node at (7,-2) {4};

\draw (4,0)--(6,2);
\draw (4,0)--(6,-2);
\draw (6,2)--(6,-2);
\path (0,0) edge [bend left] (4,0);
\path (0,0) edge [bend right] (4,0);
\end{tikzpicture}
\end{center}
\subsection{Main Results}

It is clear that $L(H)$ is a real symmetric matrix, and hence it has $n$ real eigenvalues which we denote by $\lam_0(H) \leq \lam_1(H)\le \cdots\le \lam_{n-1}(H)$.  With this in mind we can state our main result.
\begin{thm}\label{T-2}
	Let $H$ be a hypergraph with $|\e|\ge r$ for all $\e\in E(H)$.  Then for all weight vectors $x$ and $t\ge 1$,
	\[\E\l[\norm{\Av{x}{t}{}-\ol{x}}_2^2\r]\le \l(1-\f{\lam_1(H)}{r|E(H)|}\r)^t\norm{x-\ol{x}}_2^2.\]
\end{thm}

With this we will bound the rate of convergence for connected hypergraphs.  We recall that a hypergraph $H$ is \textit{connected} if for every non-empty subset $S\subsetneq V(H)$ there exists an edge $e\in E(H)$ containing a vertex in $S$ and $V(H)\sm S$.
\begin{cor}\label{C-Main}
	Let $H$ be a connected hypergraph on $n$ vertices such that $|\e|\ge r$ for all $\e\in E(H)$.  For all weight vectors $x$,  the iterated averaging process converges to its average value $\bar{x}$ as follows:
	\begin{align*}
	\E\l[\norm{\Av{x}{t}{}-\ol{x}}_2\r]\le e^{-c}\norm{x-\ol{x}}_2&\textrm{ whenever  } t \geq 2c\cdot \f{ r |E(H)|}{\lam_1(H)},\\ 
	\E\l[\norm{\Av{x}{t}{}-\ol{x}}_1\r]\le e^{-c}\norm{x-\ol{x}}_2&\textrm{ whenever }t \geq (\log(n)+2c)\cdot \f{r|E(H)|}{\lam_1(H)}.
	\end{align*}
\end{cor}

Chaterjee, Diaconis, Sly, and Zhang~\cite{CD} showed that these bounds are essentially tight when $H$ is the complete graph $K_n$.  

One can obtain concentration results for certain hypergraphs.  To this end, a hypergraph $H$ is said to be \textit{codegree regular} if there exists some $d$ with $d(u,v)=d$ for all $u\ne v$.  Examples of codegree regular hypergraphs include $K_n^{(r)}$ (the hypergraph on $\{1,\ldots,n\}$ with edge set consisting of every set of size $r$) and Steiner systems (hypergraphs where every pair is covered by exactly one edge).  Also recall that $H$ is \textit{$r$-uniform} if $|\e|=r$ for all $\e\in E(H)$.

\begin{thm}\label{T-CReg}
	Let $H$ be an $n$-vertex $r$-uniform hypergraph which is codegree regular. Then for all weight vectors $x$ and $t\ge 1$, \begin{equation}\label{E-Con}
		\E\l[\norm{\Av{x}{t}{}-\ol{x}}_2^2\r]=\l(1-\f{r-1}{n-1}\r)^t\norm{x-\ol{x}}_2^2.
	\end{equation}  Moreover, $\lim \norm{\Av{x}{t}{}-\ol{x}}_2^2/\l(1-\f{r-1}{n-1}\r)^t$ exists and is finite almost surely.
\end{thm}
The conclusions of Theorem~\ref{T-CReg} do not hold in general if $H$ is not codegree regular, see Propositions~\ref{P-Path1} and \ref{P-Path2}.

Lastly, we note that the hypergraph averaging process can be used to model other averaging processes for which our results also apply.  In particular, we define the \textit{neighborhood averaging process} as follows.  Let $G$ be a simple graph and define the neighborhood $N_G(u)$ of a vertex $u\in V(G)$ to be the set of vertices adjacent to $u$ in $G$.  For $x$ a weight vector of $G$, define the weight vector $\AvN{x}{}{G}$ by uniformly at random selecting some $u\in V(G)$, and then setting $\AvN{x}{}{G}_v=x_v$ if $v\notin N_G(u)$ and $\AvN{x}{}{G}_v=\rec{|N_G(u)|}\sum_{w\in N_G(u)} x_w$ otherwise. We iteratively define $\AvN{x}{t}{G}=\AvN{\AvN{x}{t-1}{G}}{}{G}$ and denote this simply by $\AvN{x}{t}{}$ whenever $G$ is understood.

\begin{thm}\label{T-N}
	Let $G$ be an $n$-vertex $d$-regular graph and define $\lam':=\min\{\lam_1(G),2d-\lam_{n-1}(G)\}$.  Then for all weight vectors $x$ and $t\ge 1$,  
	\[\E\l[\norm{\AvN{x}{t}{}-\ol{x}}_2^2\r]\le\l(1-\f{\lam'(2d-\lam')}{dn}\r)^t\norm{x-\ol{x}}_2^2.\]
\end{thm}
From this one can obtain bounds analogous to those of Corollary~\ref{C-Main} whenever $G$ is connected and not bipartite.

The rest of the paper is organized as follows. In Section~\ref{sec:pre} we prove some basic facts about the codegree Laplacian $L(H)$.  We then prove Theorems~\ref{T-2} and \ref{T-N} in Section~\ref{sec:converge} along with Corollary~\ref{C-Main}.  In Section~\ref{sec:concentration} we prove Theorem~\ref{T-CReg} and provide some counterexamples to concentration.  We close the paper with a number of open problems in Section~\ref{sec:conclude}.
\section{Preliminaries}\label{sec:pre}
In this section we state and prove several basic results about $L(H)$, all of which are easy generalizations of the analogous results for graphs.  To start, we show that the Raleigh quotient of $L(H)$ has a particularly nice form.  To simplify our lemmas, we adopt the convention that $\sum_{u,v}$ denotes the sums over all unordered pairs $\{u,v\}$ with $u\ne v$.

\begin{lem}\label{L-Ral}
	For $x\ne 0$ a real vector,  \[\f{x^TL(H)x}{x^Tx}=\f{\sum_{u,v} d(u,v)(x_u-x_v)^2}{\norm{x}_2^2}\]
\end{lem}
\begin{proof}
	The denominator is clear.  For the numerator, by definition we have \[(L(H)x)_u=\sum_{v\ne u} d(u,v)x_u-\sum_{v\ne u} d(u,v)x_v=\sum_{v\ne u} d(u,v)(x_u-x_v).\] Thus \[x^T L(H) x=\sum_{u} \sum_{v\ne u} d(u,v)(x_u^2-x_ux_v)\]\[=\sum_{u, v} d(u,v)(x_u^2+x_v^2-2x_ux_v)=\sum_{u, v} d(u,v)(x_u-x_v)^2.\]
\end{proof}

We recall the following well known linear algebra results, which can be found, for example, in \cite{BH}.
\begin{lem}[\cite{BH}]\label{L-CS}
	Let $M$ be a real $n\times n$ symmetric matrix.  Then $M$ has $n$ real eigenvalues $\lam_0\le \cdots \le \lam_{n-1}$ and \[\lam_0=\min_{0\ne x\in \R^n} \f{x^TMx}{x^Tx}.\] Further, any $x_0$ achieving this equality is an eigenvector corresponding to $\lam_0$ and  \[\lam_1=\min_{0\ne x\in \R^n:x\perp x_0}\f{x^TMx}{x^Tx}.\]
\end{lem}

Putting these lemmas together gives the following.
\begin{lem}\label{L-lam2}
	For all hypergraphs $H$, $\lam_0(H)=0$ and
	\[\lam_1(H)=\min_{0\ne x\in \R^n:\sum x_v=0} \f{\sum_{u,v} d(u,v)(x_u-x_v)^2}{\norm{x}_2^2}.\]
	Moreover, $\lam_1(H)>0$ if and only if $H$ is connected.
\end{lem}
\begin{proof}
	Because $L(H)$ is real symmetric, we have from Lemmas~\ref{L-CS} and \ref{L-Ral} that $\lam_0(H)$ is the minimum over non-zero real $x$ of
	\begin{equation}\label{E-Ral}
	\f{\sum_{u, v} d(u,v)(x_u-x_v)^2}{\norm{x}_2^2}
	\end{equation}
	Because the numerator and denominator of \eqref{E-Ral} are sums of squares, $\lam_0(H)\ge 0$.  Moreover, by taking $x=(1,\ldots,1)$ we see that the minimum is exactly 0 and that the all 1's vector is a corresponding eigenvector.  By Lemma~\ref{L-CS}, $\lam_1(H)$ is the minimum of \eqref{E-Ral} subject to $x\perp (1,\ldots,1)$, i.e. subject to $\sum x_u=0$.  From this it follows that $\lam_1(H)=0$ if and only if there exists a non-zero $x$ with $\sum d(u,v)(x_u-x_v)^2=0$ and $\sum x_u=0$, and we claim this happens if and only if $H$ is disconnected.  
	
	Indeed, if $H$ has a subset $S\sub V(H)$ such that every edge contains only vertices in $S$ or $V(H)\sm S$, then we can take the vector $x$ with $x_u=|S|^{-1}$ if $u\in S$ and $x_u=-|V(H)\sm S|^{-1}$ if $u\notin S$; and one can verify that $x$ satisfies the conditions above, proving that $\lam_1(H)=0$.  Conversely, if such an $x$ exists, let $S_1=\{u\in V(H):x_u\ge 0\}$ and $S_2=\{u\in V(H):x_u<0\}$.  Because $x\ne 0$ and $\sum x_u=0$ these two sets are non-empty, and hence both are proper subsets of $V(H)$.  Moreover, there exists no edge $\e\in E(H)$ with $u,v\in e$, $u\in S_1$, and $v\in S_2$, as this would imply $d(u,v)(x_u-x_v)^2>0$.  Thus $S_1\subsetneq V(H)$ shows that $H$ is disconnected as desired.
\end{proof}
\section{Bounding the Rate of Convergence}\label{sec:converge}
It turns out that one can express how much $\norm{\Av{x}{}{}}_2^2$ differs from $\norm{x}_2^2$ in a concise form.
\begin{lem}\label{L-Step}
	For any weight vector $x$ with $\sum x_u=0$,  \[\E\l[\norm{x}_2^2-\norm{\Av{x}{}{}}_2^2\r]=\rec{|E(H)|}\sum_{\e\in E(H)} \rec{|\e|}\sum_{u,v\in \e} (x_u-x_v)^2.\]
\end{lem}
\begin{proof}
	Assume the edge $\e$ is chosen in the averaging process.  Then 
	\[\norm{x}_2^2-\norm{\Av{x}{}{}}_2^2=\l(\sum_{u\in \e} x_u^2\r)-|\e|\l(\f{\sum_{v\in e} x_v}{|\e|}\r)^2\]\[=\rec{|\e|}\l(\sum_{u\in \e} (|\e|-1)x_u^2-2\sum_{u,v\in \e}x_ux_v\r)=\rec{|\e|}\sum_{u,v\in \e} (x_u-x_v)^2.\]
	As each edge is equally likely to be chosen, we conclude the result.
\end{proof}

With this we can prove our main theorem.

\begin{proof}[Proof of Theorem~\ref{T-2}]
	Let $x'$ be a weight vector.  It is not difficult to see that $\Av{x'}{t}{}-\ol{x}'=\Av{x'-\ol{x}'}{t}{}$.  Thus it is enough to prove the result for $x:=x'-\ol{x'}$, and with this we have $\sum x_u=0$ and $\ol{x}=0$.
	
	By Lemma~\ref{L-Step} and the bound $|\e|\ge r$,  \[\E[\norm{x}_2^2-\norm{\Av{x}{}{}}_2^2]=\rec{|E(H)|}\sum_{\e\in E(H)}\rec{|\e|}\sum_{u,v\in \e} (x_u-x_v)^2\]\[\ge \rec{r|E(H)|}\sum_{\e\in E(H)}\sum_{u,v\in \e}(x_u-x_v)^2=\rec{r|E(H)|}\sum_{u, v} d(u,v)(x_u-x_v)^2.\]  By Lemma~\ref{L-lam2} this quantity is at most $\rec{r|E(H)|}\lam_1(H) \norm{x}_2^2$.  By removing the determinstic value $\norm{x}_2^2$ out of the expectation, we see that the result holds for $t=1$, and the result in general follows by inductively applying the $t=1$ bound.
\end{proof}

\begin{proof}[Proof of Corollary~\ref{C-Main}]
	For the first result, we use the inequality $\E[X]\le \sqrt{\E[X^2]}$, Theorem~\ref{T-2}, and the inequality $1-\tau\le e^{-\tau}$ to conclude that 
	\begin{align*}
	\E\l[\norm{\Av{x}{t}{}-\ol{x}}_2\r]&\le \sqrt{\E[\norm{\Av{x}{t}{}-\ol{x}}_2^2]}\\ &\le \l(1-\f{\lam_1(H)}{r|E(H)|}\r)^{t/2}\norm{x-\ol{x}}_2\le \exp\l(\f{-t\lam_1(H)}{2r|E(H)|}\r)\norm{x-\ol{x}}_2.
	\end{align*}
	Plugging in $t=2c\cdot \f{r|E(H)|}{\lam_1(H)}$ gives the result, and we note that Lemma~\ref{L-lam2} and $H$ connected implies $\lam_1(H)\ne 0$ so this is well defined.
	
	For the second result, we use the Cauchy-Schwarz inequality and Theorem~\ref{T-2} to deduce that \begin{align*}\E\l[\norm{\Av{x}{t}{}-\ol{x}}_1\r]&\le \sqrt{\E\l[n\cdot \norm{\Av{x}{t}{}-\ol{x}}_2^2\r]}\\ &\le \sqrt{n}\l(1-\f{\lam_1(H)}{e(G)r}\r)^{t/2}\norm{x-\ol{x}}_2 \le \exp\l(\half\log(n)-\f{t\lam_1(H)}{2r|E(H)|}\r)\norm{x-\ol{x}}_2.\end{align*}
	Plugging in $t=(\log(n)+2c)\cdot \f{r|E(H)|}{\lam_1(H)}$ gives the result.
\end{proof}

We recall that a walk of length $k$ in a graph $G$ is a sequence of (possibly not distinct) vertices $v_0,v_1,\ldots, v_k$ such that $v_i\sim v_{i+1}$ for all $0\le i<k$.  The following standard result can be found in \cite{BH}.

\begin{lem}[\cite{BH}]\label{L-Walks}
	Let $A(G)$ be the adjacency matrix of a graph.  Then $A^k(G)_{u,v}$ is the number of walks of length $k$ from $u$ to $v$.
\end{lem}

\begin{proof}[Proof of Theorem~\ref{T-N}]
	Given a simple graph $G$, we define an auxiliary hypergraph $H_G$ by $V(H_G)=V(G)$ and $E(H_G)=\{N_G(u):u\in V(G)\}$.  Observe that $H_G$ has $n$ edges each of size $d$.  It is not difficult to see that  $\AvN{x}{t}{G}$ and $\Av{x}{t}{H_G}$ have the same distribution, so by Theorem~\ref{T-2} we have
	\begin{equation}\label{eq:N}\E[\norm{\AvN{x}{t}{G}-\ol{x}}_2^2]=\E[\norm{\Av{x}{t}{H_G}-\ol{x}}_2^2]\le \l(1-\f{\lam_1(H_G)}{dn}\r)^t \norm{x-\ol{x}}_2^2.\end{equation}
	
	Note that for $u\ne v$, the codegree $d(u,v)$ in $H_G$ is equal to the number of common neighbors of $u$ and $v$ in $G$, which is exactly the number of walks of length 2 from $u$ to $v$ in $G$.  Thus by Lemma~\ref{L-Walks} we have $L(H_G)_{u,v}=-A^2(G)_{u,v}$ for $u\ne v$ and \[L(H_G)_{u,u}=\sum_{v\ne u} A^2(G)_{u,v}=d^2-A^2(G)_{u,u},\] where this last step used that there are $d^2$ total walks of length 2 starting from $u$.  We conclude that $L(H_G)=d^2I-A^2(G)$.  Because $G$ is $d$-regular,  $A(G)=dI-L(G)$, and in particular the eigenvalues of $A^2(G)$ are exactly $(d-\lam_i(G))^2$.  Thus the eigenvalues of $L(H_G)$ will be \[d^2-(d-\lam_i(G))^2=\lam_i(G)(2d-\lam_i(G)).\] The smallest eigenvalue of $L(H_G)$ will be $0$ corresponding to $i=0$, and the second smallest eigenvalue will be \[\min\{\lam_1(G)(2d-\lam_1(G)),\lam_{n-1}(G)(2d-\lam_{n-1}(G))\}=\lam'(2d-\lam'),\]
	and this together with \eqref{eq:N} gives the result.
\end{proof}
\section{Concentration Results}\label{sec:concentration}

\begin{proof}[Proof of Theorem~\ref{T-CReg}]
	For ease of notation we assume $\ol{x}=0$, which we can do by the same argument used in the proof of Theorem~\ref{T-2}. 	Assume $d(u,v)=d$ for all $u\ne v$.  In this case $L(H)=dnI-dJ$ where $J$ is the all 1's matrix.  Thus the all 1's vector together with the $n-1$ vectors $(1,0,\ldots,0,-1,0,\ldots,0)$ form an orthogonal space of eigenvectors for $L(H)$, with the latter eigenvectors all corresponding to the eigenvalue $dn$.  In particular, every vector with $\sum x_u=0$ is an eigenvector corresponding to the eigenvalue $dn$.  Using this and Lemmas~\ref{L-Step} and \ref{L-Ral} gives that for all $x$ with $\sum x_u=0$,
	\begin{align*}\E[\norm{x}_2^2-\norm{\Av{x}{}{}}_2^2]&=\rec{r|E(H)|}\sum_{u\ne v} d(u,v)(x_u-x_v)^2\\ &=\rec{r|E(H)|}\cdot x^TL(H)x=\f{nd}{r|E(H)|}\norm{x}_2^2.\end{align*}
	Pulling out the deterministic value $\norm{x}_2^2$ gives \[\E[\norm{\Av{x}{}{}}_2^2]=\l(1-\f{nd}{r|E(H)|}\r)\norm{x}_2^2.\]  
	
	To complete the proof of \eqref{E-Con} when $t=1$, we must show that $|E(H)|=\f{n(n-1)d}{r(r-1)}$.  To do this, we count the pairs $(\{u,v\},e)$ with $u\ne v$ and $u,v\in e$ in two ways.  We can first choose the pair $\{u,v\}$ in ${n\choose 2}$ ways and then the edge in $d$ ways, or we could choose the edge first in $|E(H)|$ ways and then a pair it contains in ${r\choose 2}$ ways.  This implies that ${n\choose 2}d=|E(H)|{r\choose 2}$, giving the desired result.  The result for general $t$ follows by inductively applying the $t=1$ case.
	
	For the concentration result, define $S^t(x)=(1-\f{r-1}{n-1})^{-t}\cdot \Av{x}{t}{}$, which in particular implies $S^t(x)=(1-\f{r-1}{n-1})^{-1}\cdot  \Av{S^{t-1}(x)}{}{}$.  This together with \eqref{E-Con} implies that given $S^{t-1}(x)$, we have \[\E\l[\norm{S^t(x)}_2^2\r]=\f{\E\l[\norm{\Av{S^{t-1}(x)}{}{}}_2^2\r]}{(1-\f{r-1}{n-1})}=\norm{S^{t-1}(x)}_2^2.\]
	Thus $\norm{S^t(x)}_2^2$ is a non-negative martingale, so its limit exists and is finite almost surely.
\end{proof}

We close this section with some examples where $H$ is not codegree regular and where the conclusions of Theorem~\ref{T-CReg} fail to hold. Here and throughout when we consider weight vectors $x$ on the path graph $P_n$, we let $x_1,x_n$ be the weights of the endpoints of the path.

\begin{prop}\label{P-Path1}
	Let $x$ be the weight vector of $P_3$ with $x=(1,-\half,-\half)$.  Then for all $t\ge 1$,
	\[\Pr\l[\norm{\Av{x}{t}{}}_2^2\ge 2^{-t/2}\r]\ge \half.\]
\end{prop}

In contrast, Theorem~\ref{T-CReg} would predict that $\E[\norm{\Av{x}{t}{}}_2^2](2-\ep)^t$ tends to 0 for all $\ep>0$ if $P_3$ were codegree regular.

\begin{proof}
Let $D(t)$ denote the number of $s$ with $1\le s\le t$ such that $\Av{x}{s}{}\ne \Av{x}{s-1}{}$.  One can prove by induction that if $D(t)$ is even then $\Av{x}{t}{}=(2^{- D(t)},-2^{-D(t)-1},-2^{-D(t)-1})$ and otherwise $\Av{x}{t}{}=(2^{-D(t)-1},2^{-D(t)-1},-2^{-D(t)})$.  In particular, given $D(t)$ we have $\norm{\Av{x}{t}{}}_2^2\ge 2^{-D(t)}$.  Thus it is enough to show that $\Pr[D(t)\le t/2]\ge \half$.  It is not difficult to see that the distribution of $D(t)$ is binomial with $t$ trials and probability $\half$ of successes (each round has probability $\half$ of choosing the one edge that will change $\Av{x}{t}{}$).  Thus this statement is equivalent to showing that $\sum_{i=0}^{t/2} {t\choose i}\ge 2^{t-1}$, which is easy to prove by the symmetry of the binomial coefficients.
\end{proof}
 A similar example shows that there exist $x$ such that $\Av{x}{t}{}$ can exhibit different long term behaviors.

\begin{prop}\label{P-Path2}
	Let $x$ be the weight vector of $P_3$ with $x=(1,-1,0)$. Then for all $t\ge 1$, \[\Pr\l[\norm{\Av{x}{t}{}}_2^2=0\r]= \half,\]\[\Pr\l[\norm{\Av{x}{t}{}}_2^2\ge 2^{-t}\r]=\half.\]
\end{prop}

\begin{proof}
	Let $x=(1,-1,0)$. With probability $\half$ the edge $\{1,2\}$ is chosen first, and then for all $t\ge 1$ we have $\Av{x}{t}{}=0$.  If $\{2,3\}$ is chosen first then $\Av{x}{1}{}=(1,-\half,-\half)$.  In this case, the same reasoning as in the previous proof shows that $\norm{\Av{x}{t+1}{}}_2^2\ge 2^{-D(t)}$ with $D(t)$ a random variable that is at most $t$ (we shift $t$ by 1 here because this is the second step of this random process).  In particular, we have $\norm{\Av{x}{t}{}}_2^2\ge 2^{-t}$ for all $t$ in this case.
\end{proof}

\section{Concluding Remarks}\label{sec:conclude}
For ease of presentation, whenever $H$ is understood we define 
\[\Del_1(t,x)=\norm{\Av{x}{t}{H}-\ol{x}}_1.\]
The second half of Corollary~\ref{C-Main} shows that in expectation $\Del_1(t,x)$ will be small provided $t\approx \f{r |E(H)|\log(n)}{\lam_1(H)}$, and this bound is essentially tight for the complete graph $K_n$ due to work of Chaterjee, Diaconis, Sly, and Zhang~\cite{CD}.  It is not clear whether these bounds are tight for all hypergraphs, or even for all graphs, and we ask the following somewhat vague question.

\begin{quest}
	When are the bounds in Corollary~\ref{C-Main} essentially tight?
\end{quest}
We give two concrete conjectures in this direction. Let $S_n$ be the star graph on $n+1$ vertices.  Note that $|E(S_n)|=n$ and $\lam_1(G)=1$, so Corollary~\ref{C-Main} shows that for any $x$ we have $\Del_1(t,x)\approx 0$ in expectation whenever $t\approx 2n\log(n)$.  We suspect that this is tight.

\begin{conj}
	Let $x$ be the weight vector on $S_n$ which gives weight $1-\rec{n+1}$ to the central vertex and weight $-\rec{n+1}$ to every other vertex.  Then for $t=o(n\log n)$ we have \[\E[\Del_1(t,x)]\sim\norm{x-\ol{x}}_2\]
\end{conj}
Figure~\ref{FS} shows a plot of $\Del_1(t,x)$ for this $x$ and $S_{1000}$.  Note that in this case $2n\log(n)\approx 13,800$, and it does appear to take this long for $\Del_1(t,x)$ to converge to 0.

\begin{figure}[h]
	\centering
	\includegraphics[width=0.8\textwidth]{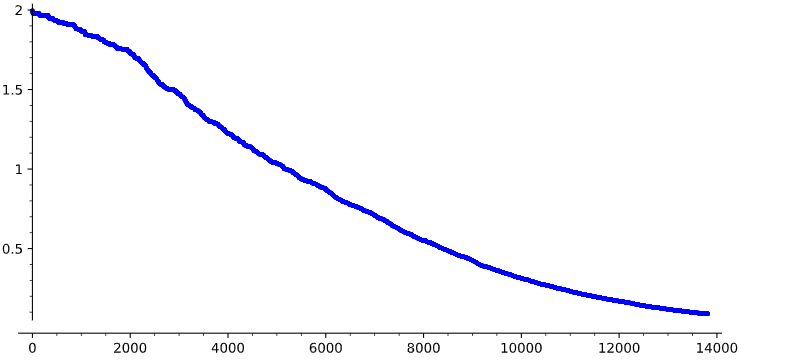}
	\caption{$\Del_1(t,x)$ for $S_{1000}$ with $x_1=1-\f{1}{1001}$ and $x_{i}=\f{-1}{1001}$ for all other $i$.}\label{FS}
\end{figure}

On the other hand, we do not expect the bound of Corollary~\ref{C-Main} to be tight for paths. If $P_n$ is the path graph on $n$ vertices, then $\lam_1(P_n)=2-2\cos(\pi/n)=\Theta(n^{-2})$ and Corollary~\ref{C-Main} implies that $\Del_1(t,x)$ will be small for $t=\Theta(n^3\log(n))$.   Figure~\ref{F1} gives a plot of $\Del_1(t,x)$ when $x$ is the weight vector of $P_{40}$ taking value $\pm 1$ on each endpoint of the path and 0 everywhere else; and Figure~\ref{F2} shows a plot when $x$ has weight $1-1/40$ on one endpoint and $-1/40$ on every vertex.  Note that both processes seem to converge within $n^3=64,000$ steps. These results motivate us to conjecture that  the $\log(n)$ term in Corollary~\ref{C-Main} is not necessary for the path.

\begin{figure}[h]
	
	\centering
	\includegraphics[width=0.8\textwidth]{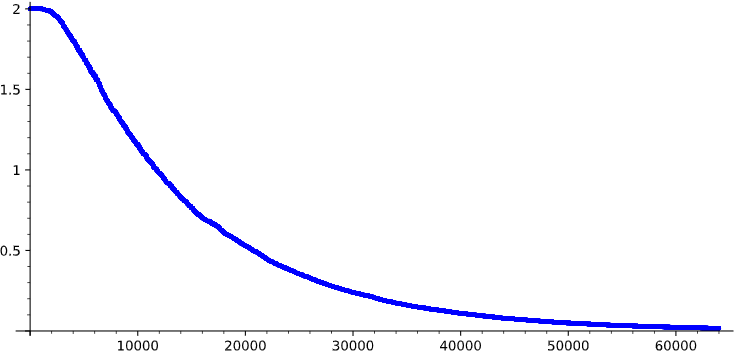}
	\caption{$\Del_1(t,x)$ for $P_{40}$ with $x_1=1,\ x_{40}=-1 $ and $0$ elsewhere.}\label{F1}
\end{figure}

\begin{figure}[h]
	
	\centering
	\includegraphics[width=0.8\textwidth]{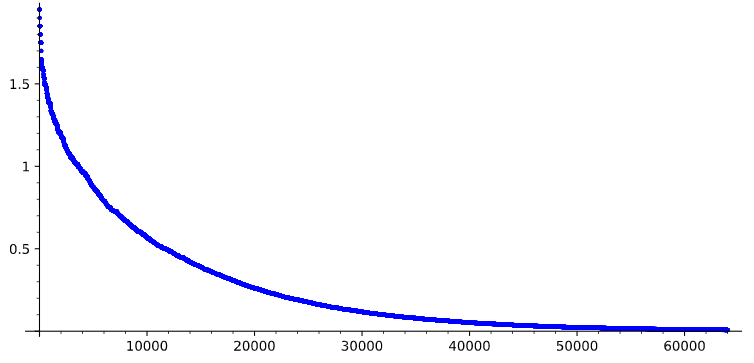}
	\caption{$\Del_1(t,x)$ for $P_{40}$ with $x_1=1-\rec{40}$ and $x_{i}=\f{-1}{40}$ for all other $i$.}\label{F2}
\end{figure}

\begin{conj}
	For all $c>0$, there exists a constant $M=M(c)$ such that for all weight vectors $x$ of $P_n$, we have \[\E[\Del_1(Mn^3,x)]\le e^{-c}\norm{x-\ol{x}}_2.\]
\end{conj}

We now turn our attention to the neighborhood averaging process.  By adapting the proof of Theorem~\ref{T-N} one can obtain bounds on the convergence of $\AvN{x}{t}{G}$ for all graphs $G$ in terms of $\lam_1(H_G)$, where $H_G$ is the auxiliary hypergraph introduced in the proof of Theorem~\ref{T-N}; and more precisely one can show that $\AvN{x}{t}{G}$ will always converge to $\ol{x}$ if and only if $G$ is connected and not bipartite.

Unfortunately, it is impossible to express $\lam_1(H_G)$ in terms of $\lam_i(G)$ for general graphs $G$.  Indeed, it is well known that the eigenvalues of the Laplacian $L(G)$ can not detect whether $G$ is bipartite in general, so in particular it can not detect whether $\AvN{x}{t}{G}$ will always converge to $\ol{x}$.  However, it may be possible to express $\lam_1(H_G)$ in terms of eigenvalues of a different matrix associated to $G$.  The most natural candidate would be the normalized Laplacian $\c{L}(G)$ since its eigenvalues can detect whether $G$ is connected and bipartite in general; see the survey of Butler and Chung~\cite{BC} for more on the normalized Laplacian.  With this in mind we pose the following question.

\begin{quest}
	For any connected and not bipartite graph $G$, can one  bound the convergence of $\AvN{x}{t}{G}$ in terms of the eigenvalues of the normalized Laplacian matrix $\c{L}(G)$?
\end{quest}

\section*{Acknowledgments}
The author would like to thank Fan Chung for suggesting this problem.  We thank her and a referee for helpful comments on earlier drafts of this paper.  This material is based upon work supported by the National Science Foundation Graduate Research Fellowship under Grant No. DGE-1650112.

\bibliographystyle{abbrv}
\bibliography{Average}

\end{document}